\documentclass[12pt]{article}

\usepackage{graphicx}
\usepackage{lineno,hyperref}
\usepackage{amsthm,amsfonts,amssymb,latexsym,amsmath}
\usepackage{float}
\usepackage{slashbox,tikz}
\newcommand*\circled[1]{\tikz[baseline=(char.base)]{
		\node[shape=circle,draw,inner sep=2pt] (char) {#1};}}

\modulolinenumbers[5]
\newtheorem{theorem}{Theorem}

\newtheorem{remark}{Remark}

\newtheorem{proposition}{Proposition}

\newcommand{\be}{\begin{eqnarray}}
\newcommand{\ee}{\end{eqnarray}}
\newcommand{\nn}{{\nonumber}}

\newcommand{\LL}{{\cal L}}
\newcommand{\ignore}[1]{}
\newcommand{\sFB}{{\cal FB}}
\newcommand{\diam}{{\rm diam}}
\newcommand{\sLB}{{\cal LB}}
\newcommand{\sALB}{{\cal ALB}}

\begin{document}
	\title{Euler numbers and diametral paths in Fibonacci cubes, Lucas cubes and Alternate Lucas cubes}
	\author{
		\"Omer E\u{g}ecio\u{g}lu
		\thanks{Department of Computer Science, University of California Santa Barbara, Santa Barbara, California 93106, USA. email: omer@cs.ucsb.edu}
		\quad
		Elif Sayg{\i}
		\thanks{Department of Mathematics and Science Education, Hacettepe University, 06800, Ankara, Turkey. email: esaygi@hacettepe.edu.tr}
		\quad
		Z\"ulf\"ukar Sayg{\i}
		\thanks{Department of Mathematics, TOBB University of Economics and Technology, 06560, Ankara, Turkey. email: zsaygi@etu.edu.tr}
	}
	
	
	\maketitle
	
	\begin{abstract}
		The diameter of a graph is the maximum distance between 
		pairs of vertices in the graph. A pair of vertices whose distance is equal to 
		its diameter are called diametrically opposite vertices. The collection of shortest paths between diametrically 
		opposite vertices are referred as diametral paths. In this work, 
		we enumerate the number of diametral paths for Fibonacci cubes, Lucas cubes 
		and Alternate Lucas cubes. We present bijective proofs that
		show that  these numbers are related to alternating permutations and 
		are enumerated by Euler numbers. 
	\end{abstract}

	\maketitle
	
\linenumbers

\section{Introduction}\label{section.intro}	
Given a connected graph $G=(V,E)$, one of the basic problem is to enumerate the number 
of shortest paths between pairs of vertices in $G$. 
The solution to this problem
provides an important topological property of an interconnection network, 
in terms of its connectivity, fault-tolerance, 
communication expense \cite{IPL2002} and 
has important applications such as for counting minimum $(s,t)$-cut in 
planar graphs and route guidance systems \cite{TCS12}.

The so-called
single-source shortest paths problem consists of finding the shortest paths between 
a given vertex and all other vertices in the graph. One can solve this problem 
by using the algorithms such as Breadth-First-Search for unweighted graphs or 
Dijkstra's algorithm \cite{Dijkstra}. 
Similarly, Dijkstra's algorithm can be used to solve the single-pair shortest paths problem 
in a weighted, directed graph with nonnegative weights.

The process of finding all shortest paths between a pair of 
vertices in a graph is another problem. This can be considered a 
search for the most efficient routes through 
the graph. In \cite{CR_NP}, it is proved that finding the number 
of shortest paths in a general graph is NP-hard.

For planar graphs with $m$ vertices an oracle is presented in \cite{ISAAC18} to 
find the number of shortest paths for a given pair of vertices whose time complexity is 
$O(\sqrt{m})$ with $O(m^{1.5})$ space. This approach is 
improved in \cite{TCS22} and a new oracle for counting shortest paths in 
planar graphs is presented, where Voronoi diagrams are used 
to speed up the query time.

In the literature, the problem of enumerating the number of 
shortest paths have been considered for some special families of graphs. Explicit formulas
have been obtained for the hexagonal network \cite{GSSS}, the star 
graph \cite{IR}, the $(n, k)$-star graph \cite{CGLQS,CQS} and 
the arrangement graph \cite{CGQS}. In an $(n, k)$-star graph the number of 
shortest paths is enumerated by counting the minimum factorizations of 
a permutation in terms of the transpositions corresponding to edges in that graph \cite{CQS}. 
For  the arrangement graph this number is enumerated by establishing a bijection 
between these shortest paths and a collection of ordered forests 
of certain bi-colored trees \cite{CGQS}.

The distance $d(u,v)$ between two vertices $u,v\in V$ is 
the number of edges in a shortest path between $u$ and $v$. 
The {\em diameter} of $G$ is defined as the maximum distance between pairs of vertices in $V$ 
and is denoted by $\diam(G)$.

A pair of vertices $u,v\in V$ with $d(u, v) = \diam(G)$, are called {\em diametrically opposite} 
vertices. The collection of shortest paths between diametrically opposite vertices are 
referred to as {\em diametral paths}. For a pair 
of diametrically opposite vertices $u,v\in V$ we let $c(u,v;G)$ 
denote the number of diametral paths from $u$ to $v$ in $G$.

As an example, for the $n$-dimensional hypercube $Q_n$ 
the number of diametral 
paths between any diametrically opposite pair $u$ and $v$ 
can be enumerated by establishing a bijection between these shortest paths and the 
permutations on $n$ symbols, so that 
$$
c(u,v;Q_n)=n! \,.
$$

In this paper we enumerate the number of diametral paths for three special subgraphs of 
hypercube graphs, namely Fibonacci cubes \cite{Hsu}, Lucas cubes \cite{Lucas} 
and Alternate Lucas cubes \cite{ALucas}.
We present bijective proofs of our results.
Surprisingly, these numbers are  related to 
alternating permutations and are enumerated by Euler numbers. 
 
\section{Preliminaries}\label{section.prelim}
We let $[n] = \{1,2, \ldots, n\}$.
The $n$-dimensional hypercube $Q_n$ is the graph defined on the vertex set $B_n$, where
$$B_n=\{b_1b_2\ldots b_n\mid b_i\in\{0,1\},\  i \in [n] \} \, .$$
Two vertices $u,v\in B_n$ are adjacent if and only if the Hamming distance $d(u,v)=1$, that is, $u$ and $v$ differ in exactly one coordinate. For convenience, $Q_0=K_1$. It is clear from the definition that 
$\diam(Q_n)=n$ and for any vertex $u\in B_n$ there exist a unique vertex $\bar{u}\in B_n$ such that $d(u,\bar{u})=n$, where $\bar{u}$ denotes the complement of the binary string of $u$.

For $n\ge 1$, let 
$$\sFB_n=\{b_1b_2\ldots b_n\in B_n\mid b_i\cdot b_{i+1}=0,\ i \in [n-1]\} \, .$$
The $n$-dimensional Fibonacci cube $\Gamma_n$ ($n\ge1$) is an induced subgraph of $Q_n$ with 
vertex set $\sFB_n$. We take $\Gamma_0=K_1$.
Similarly, for $n\ge 1$, let 
$$\sLB_n=\{b_1b_2\ldots b_n\in \sFB_n\mid b_1\cdot b_n=0\}$$
and for $ n \geq 3$,
$$\sALB_n=\{b_1b_2\ldots b_n\in \sFB_n\mid b_n\cdot b_{n-2}=0\} \, .$$
The $n$-dimensional Lucas cube $\Lambda_n$ and Alternate 
Lucas cube $\LL_n$ are defined as the induced subgraphs of $\Gamma_n\subseteq Q_n$ 
and with sets $\sLB_n$ and $\sALB_n$, respectively. 

$Q_n$ has a useful decomposition in which its vertex set is partitions into two sets 
$B_n=0B_{n-1}\cup 1B_{n-1}$,
where $0B_{n-1}$ denotes the vertices that start with a $0$ and $1B_{n-1}$ denotes 
the vertices that start with a $1$.
Using this decomposition we can write
$$Q_n=0Q_{n-1}+1Q_{n-1}$$
where $0Q_{n-1}$ and $1Q_{n-1}$ denote the induced subgraphs of $Q_n$ with vertex 
sets $0B_{n-1}$ and $1B_{n-1}$ respectively, and $+$ denotes 
the perfect matching between $0Q_{n-1}$ and $1Q_{n-1}$.
Similarly, we have the following fundamental decompositions for Fibonacci cubes, Lucas cubes and 
Alternate Lucas cubes:
\be
\Gamma_n=0\Gamma_{n-1}+10\Gamma_{n-2} \, ,
\nn\ee
where there is a perfect matching between $10\Gamma_{n-2}$ and $00\Gamma_{n-2}\subset 0\Gamma_{n-1}$,
\be
\Lambda_n=0\Gamma_{n-1}+10\Gamma_{n-3}0 \, ,
\nn\ee
where there is a perfect matching between $10\Gamma_{n-3}0$ and $00\Gamma_{n-3}0\subset 0\Gamma_{n-1}$,
\be
\LL_n=0\LL_{n-1}+10\LL_{n-2} \, ,
\nn\ee
where there is a perfect matching between $10\LL_{n-2}$ and $00\LL_{n-2}\subset 0\LL_{n-1}$.

\subsection{Euler numbers}\label{section.Euler}
Following \cite{Stanley}, a permutation $\sigma =  \sigma_1 \sigma \ldots \sigma_n$ of 
$[n]$ is 
{\em alternating} if 
$\sigma_1 > \sigma_2, < \sigma_3 > \sigma_4 < \cdots$. In other words, $\sigma_i < \sigma_{i+1} $ for $i$ 
even and $a_i > a_{i+1}$ for $i$ odd.
$\sigma$ is {\em reverse alternating} if $ \sigma_1 < \sigma_2 > \sigma_3 < \sigma_4 \cdots$.
Let $E_n$ denote the number of alternating permutations of $[n]$ with $E_0 =1$. 
These are known as the {\em Euler numbers}. The number of reverse alternating permutations 
of $[n]$ is also given by 
$E_n$.

By a result of D{\'e}sir{\'e} Andr{\'e} \cite{Andre}, we have
$$
2 E_{n+1} = \sum_{k=0}^n { n \choose k} E_k E_{n-k} \, ,
$$
and the exponential generating function of the sequence of Euler
numbers is given by 
\begin{eqnarray*}
	\sum_{n \geq 0} E_n \frac{x^n}{n!} & =& \sec x + \tan x \\
	&=& 1 + x + \frac{x^2}{2!} + 2 \frac{x^3}{3!} + 5 \frac{x^4}{4!} + 16 \frac{x^5}{5!} + 61 \frac{x^6}{6!}+\cdots
\end{eqnarray*}
First few terms of Euler numbers (sequence A000111 in the OEIS \cite{oeis})
are
$E_0 = 1 $,
$E_1 = 1 $,
$E_2 = 1 $,
$E_3 = 2 $,
$E_4 = 5 $,
$E_5 = 16 $,
$E_6 = 61 $. 

\section{Calculation for the Fibonacci cubes}\label{section.fibo}	
In this section we determine the number of 
diametral paths in $\Gamma_n$. 
Since $n= \diam(\Gamma_n)$, these paths are of length $n$.
$\Gamma_n$ is an induced subgraph of $Q_n$ with vertex set $\sFB_n$. We have the following 
easy result.

\begin{proposition}\label{prop.Fibo.max.distance.vertices}
	There is a unique pair of diametrically opposite vertices in $\Gamma_n$. They are
	\begin{itemize}
		\item[(i)] $u=(01)^{\frac{n}{2}}$ and $v=(10)^{\frac{n}{2}}$ if $n$ is even,
		\item[(ii)] $u=(01)^{\frac{n-1}{2}} 0 $ and  $v=(10)^{\frac{n-1}{2}} 1 $ if $n$ is odd.
	\end{itemize}
\end{proposition}
Even though $ \Gamma_n$ is undirected, 
here we view the edges on each such path to be 
directed from $u$ to $v$. By direct inspection we have $c(u,v;\Gamma_1)=1$, $c(u,v;\Gamma_2)=1$, $c(u,v;\Gamma_3)= 2$, $c(u,v;\Gamma_4)=5$, $c(u,v;\Gamma_5) = 16$.
For a path
$$
u = s_0  \rightarrow s_1 \rightarrow \cdots \rightarrow s_n = v~,
$$
each vertex $s_{i+1}$ is obtained from the vertex $s_i$ by flipping a 0 to a 1, 
or a 1 to a 0, with the proviso that no consecutive 1s
appear in any $s_i$. 
We see in particular that $c(u,v;\Gamma_3)= 2$ as there are 
two paths of length 3 from $u$ to $v$ when $n=3$ as shown in the Table \ref{table0}.

\begin{table}[H]
\begin{center}
	\begin{tabular}{|r|ccc|}
		\hline
		Step & $~b_1~$ & $~b_2~$ & $~b_3~$ \\ \hline 
		$v=s_3$ & 1&0&1 \\ 
		$s_2$   & 1&0&0 \\ 
		$s_1$   & 0&0&0 \\ 
		$u=s_0$ & 0&1&0\\ \hline
	\end{tabular}
	\quad
	\begin{tabular}{|r|ccc|}
		\hline
		Step & $~b_1~$ & $~b_2~$ & $~b_3~$ \\ \hline 
		$v=s_3$ & 1&0&1 \\ 
		$s_2$   & 0&0&1 \\ 
		$s_1$   & 0&0&0 \\ 
		$u=s_0$ & 0&1&0\\ \hline
	\end{tabular}
\end{center}
	\caption{Two different paths from $u=010$ to $v=101$ in $\Gamma_3$.}\label{table0}
\end{table}

Here we write $u$ in the bottom most row. The $i$th step 
shows the string $s_i$ after $i$ edges on the path have 
been traversed. Note that in this representation, the path proceeds from bottom up 
and the row indices are increasing from bottom up as well.

By using this representation we give a bijective proof that the
sequence of the numbers of diametral paths in Fibonacci cubes is precisely
the sequence of Euler numbers.
\begin{theorem}\label{theorem_cn_Fibo}
		Let $u,v\in\Gamma_n$ such that $d(u,v)=n$. Then for $ n \geq 1$, we have $$c(u,v;\Gamma_n) = E_n~,$$
		where $E_n$ is the $n$th Euler number.
\end{theorem}
\begin{proof}
	We give a bijection between paths of length $n$ from $u$ to $v$ in $\sFB_n$ and alternating permutations $\sigma$ of 
	$[n]$.
	The bijection is best communicated by an example. Suppose $ n = 8$ and we are given the path 
	from $u = 01010101$ to $ v = 10101010$ whose steps are shown in Table \ref{table1}.
	
	\begin{table}[H]
		\begin{center}
			\begin{tabular}{|r|cccccccc|}
				\hline
				Step & $~b_1~$ & $~b_2~$ & $~b_3~$ & $~b_4~$ & $~b_5~$ & $~b_6~$ & $~b_7~$ & $~b_8~$ \\ \hline 
				$v=s_8$ & 1&0&1&0&1&0&1&0  \\ 
				$s_7$   & 1&0&1&0&1&0&0&0\\ 
				$s_6$   & 1&0&0&0&1&0&0&0\\
				$s_5$   & 1&0&0&0&1&0&0&1\\ 
				$s_4$   & 0&0&0&0&1&0&0&1 \\ 
				$s_3$   & 0&1&0&0&1&0&0&1 \\ 
				$s_2$   & 0&1&0&0&0&0&0&1 \\ 
				$s_1$   & 0&1&0&0&0&1&0&1 \\ 
				$u=s_0$ & 0&1&0&1&0&1&0&1 \\ \hline
			\end{tabular}
		\end{center}
		\caption{A path from diametrically opposite vertices $u=(01)^4$ to $v=(10)^4$ in $\Gamma_8$.}\label{table1}
	\end{table}
	
	As the first step, we mark the first appearance of 1 as we go up the table in every column with 
	an odd index. In Table \ref{table2} these entries are circled.
	\begin{table}[H]
		\begin{center}
			\begin{tabular}{|r|cccccccc|}
				\hline
				Step & $~b_1~$ & $~b_2~$ & $~b_3~$ & $~b_4~$ & $~b_5~$ & $~b_6~$ & $~b_7~$ & $~b_8~$ \\ \hline 
				$v=s_8$ & 1&0&1&0&1&0&{\bf \circled{1}}&0  \\ 
				$s_7$   & 1&0&{\bf \circled{1}}&0&1&0&0&0\\ 
				$s_6$   & 1&0&0&0&1&0&0&0\\ 
				$s_5$   & {\bf \circled{1}} &0&0&0&1&0&0&1\\ 
				$s_4$   & 0&0&0&0&1&0&0&1 \\ 
				$s_3$   & 0&1&0&0&{\bf \circled{1}}&0&0&1 \\ 
				$s_2$   & 0&1&0&0&0&0&0&1 \\ 
				$s_1$   & 0&1&0&0&0&1&0&1 \\ 
				$u=s_0$ & 0&1&0&1&0&1&0&1 \\ \hline
			\end{tabular}
		\end{center}
		\caption{First appearance of 1 as we go up in every column with an odd index is marked in the path from $u=(01)^4$ to $v=(10)^4$ in $\Gamma_8$.}\label{table2}
	\end{table}
	Next, we mark the first appearance of 0 as we go up the table in every column with 
	an even index. Circling these entries gives Table \ref{table3}.
	\begin{table}[H]
		\begin{center}
			\begin{tabular}{|r|cccccccc|}
				\hline
				Step & $~b_1~$ & $~b_2~$ & $~b_3~$ & $~b_4~$ & $~b_5~$ & $~b_6~$ & $~b_7~$ & $~b_8~$ \\ \hline 
				$v=s_8$ & 1&0&1&0&1&0&{\bf \circled{1}}&0  \\ 
				$s_7$   & 1&0&{\bf \circled{1}}&0&1&0&0&0\\ 
				$s_6$   & 1&0&0&0&1&0&0&{\bf \circled{0}}\\ 
				$s_5$   & {\bf \circled{1}}&0&0&0&1&0&0&1\\ 
				$s_4$   & 0&{\bf \circled{0}}&0&0&1&0&0&1 \\ 
				$s_3$   & 0&1&0&0&{\bf \circled{1}}&0&0&1 \\ 
				$s_2$   & 0&1&0&0&0&{\bf \circled{0}}&0&1 \\ 
				$s_1$   & 0&1&0&{\bf \circled{0}}&0&1&0&1 \\ 
				$u=s_0$ & 0&1&0&1&0&1&0&1 \\ \hline
			\end{tabular}
		\end{center}
		\caption{First appearance of 0/1 as we go up in every column with an even/odd index is marked in the path from $u=(01)^4$ to $v=(10)^4$ in $\Gamma_8$.}\label{table3}
	\end{table}
	After this
	we record the corresponding step number in each column. For instance column 1 gives 5, column 2 gives 4, 
	etc. by reading the indices of the corresponding rows.
	The resulting alternating permutation is below:
	\begin{center}
		~~~~~~~~~~5 ~~~~ 4 ~~~~ 7 ~~~~ 1 ~~~~ 3 ~~~~ 2~~~~  8~~~~  6
	\end{center}
	
	These steps are reversible. Suppose this time that $n= 7 $ and we are given the alternating 
	permutation $ 3 ~1~ 6~ 4~ 7~ 2~ 5 $.
	We construct Table \ref{table4} in which the odd numbered columns 1, 3, 5, 7 are assigned the label 
	1 in the rows 3, 6, 7, 5, which are the entries in the odd positions of the given permutation. 
	The even numbered columns 2, 4, 6 are assigned the label 0 in the rows 1,4, 2, which are the entries in the 
	even indexed positions of the given permutation.
	
	\begin{table}[H]
		\begin{center}
			\begin{tabular}{|r|ccccccc|}
				\hline
				Step & $~b_1~$ & $~b_2~$ & $~b_3~$ & $~b_4~$ & $~b_5~$ & $~b_6~$ & $~b_7~$  \\ \hline 
				$v=s_7$  & ~&~&~&~&{\bf \circled{1}}&~&~\\ 
				$s_6$   & ~&~&{\bf \circled{1}}&~&~&~&~\\ 
				$s_5$   & ~&~&~&~&~&~&{\bf \circled{1}}\\ 
				$s_4$   & ~&~&~&{\bf \circled{0}}&~&~&~ \\ 
				$s_3$   & {\bf \circled{1}}&~&~&~&~&~&~ \\ 
				$s_2$   & ~&~&~&~&~&{\bf \circled{0}}&~ \\ 
				$s_1$   & ~&{\bf \circled{0}}&~&~&~&~&~ \\ 
				$u=s_0$ & &~ ~&~&~&~&~&~ \\ \hline
			\end{tabular}
		\end{center}
		\caption{First appearance of 0/1 in every column with an even/odd index in the path from $u=(01)^30$ to $v=(10)^31$ in $\Gamma_7$ corresponding to the 
alternating permutation $ 3 ~1~ 6~ 4~ 7~ 2~ 5 $.}\label{table4}
	\end{table}
	
	Now we fill in the odd indexed columns of this matrix by 0, up to the marked 1 in the column, followed by 0s all the way up; 
	and we fill the even indexed columns by 1 up to the marked 0 in the column, followed by 1s 
	all the way up. This results in the path of length $n=7$ from $u$ to $v$ shown in Table \ref{table5} corresponding to the 
alternating permutation $ 3 ~1~ 6~ 4~ 7~ 2~ 5 $.
	\begin{table}[H]
		\begin{center}
			\begin{tabular}{|r|ccccccc|}
				\hline
				Step & $~b_1~$ & $~b_2~$ & $~b_3~$ & $~b_4~$ & $~b_5~$ & $~b_6~$ & $~b_7~$  \\ \hline 
				$v=s_7$  & 1&0&1&0&{\bf \circled{1}}&0&1\\ 
				$s_6$   & 1&0&{\bf \circled{1}}&0&0&0&1\\ 
				$s_5$   & 1&0&0&0&0&0&{\bf \circled{1}}\\ 
				$s_4$   & 1&0&0&{\bf \circled{0}}&0&0&0 \\ 
				$s_3$   & {\bf \circled{1}}&0&0&1&0&0&0 \\ 
				$s_2$   & 0&0&0&1&0&{\bf \circled{0}}&0 \\ 
				$s_1$   & 0&{\bf \circled{0}}&0&1&0&1&0 \\ 
				$u=s_0$ & 0 &1&0&1&0&1 &0 \\ \hline
			\end{tabular}
		\end{center}
		\caption{The path from the diametrically opposite vertices $u=(01)^30$ to $v=(10)^31$ in $\Gamma_7$ corresponding to the 
alternating permutation $ 3 ~1~ 6~ 4~ 7~ 2~ 5 $.}\label{table5}
	\end{table}
	
	Considering now the general case, we see that going from $u$ to $v$ in $n$ steps, every bit in 
	$u$ has to change exactly once. This means that the row indices of the marked entries in the matrix
	in Table \ref{table3} is a permutation $\sigma$ of $[n]$. 
	Now consider an element $ t= \sigma_i$ with odd $i$ with $1< i < n$.
	This means that in step $t$ of the path, that is in $ s_t$, the entry in the $i$th column goes from 
	0 to 1. But since all of the vertices that appear in the table as rows are Fibonacci strings. This means 
	that in $s_{t-1}$  the entries in columns $i-1$ and $i+1$ which are 
	adjacent to the entry at column $i$ must already be 0.
	Therefore 
	these entries were flipped from 1 to 0 in earlier steps. It follows that  $\sigma_i > \sigma_{i+1}$ and 
	$\sigma_i > \sigma_{i-1}$. The two extreme cases with $i=1$ and $i=n$ are handled the same way. Therefore $ \sigma$ is 
	an alternating permutation. The other direction is proved similarly.
\end{proof}

\section{Calculation for the Lucas cubes}
It is shown in \cite{Lucas} that
$$\diam(\Lambda_n)=
\left\{\begin{array}{ll}
	n & \mbox{for $n$ even} \, ,\\
	n-1 & \mbox{for $n$ odd} \, .
\end{array}\right.
$$ 
We have 
\begin{proposition}\label{prop.Lucas.max.distance.vertices}
	The number of diametrically opposite pair of vertices in $\Lambda_n$ is 1 if $n$ is even and $n$ if $n$ is odd. They are
	\begin{itemize}
		\item $(01)^{\frac{n}{2}}$ and $(10)^{\frac{n}{2}}$ if $n$ is even,
		\item cyclic shifts of the pair $0(01)^{\frac{n-1}{2}}$ and $0(10)^{\frac{n-1}{2}}$ if $n$ is odd.
	\end{itemize}
\end{proposition}
\begin{remark}
	Note that there is a typo in \cite[Proposition 1]{Lucas}. For $n$ odd, 
the number of pairs of vertices in $\Lambda_n$ at distance equal to the diameter is $n$, not $n-1$.
\end{remark}
Similar to the proof of Theorem \ref{theorem_cn_Fibo} we obtain the following result for $\Lambda_n$.
\begin{theorem}\label{theorem_cn_Lucas}
	Let $u,v\in\Lambda_n$ such that $d(u,v)=\diam(\Lambda_n)$. Then for $ n \geq 2$, we have
	$$c(u,v;\Lambda_n)=
	\left\{\begin{array}{ll}
		\frac{n}{2}E_{n-1} & \mbox{for $n$ even} \, ,\\
		E_{n-1} & \mbox{for $n$ odd} \, .
	\end{array}\right.$$
\end{theorem}
\begin{proof}
 Assume first that $n$ is even. By Proposition \ref{prop.Lucas.max.distance.vertices} 
we only need to consider the vertices $u=(01)^{\frac{n}{2}}$ and $v=(10)^{\frac{n}{2}}$. 
Mimicking the bijective proof of Theorem  \ref{theorem_cn_Fibo}, we arrive at
permutations $\sigma$ of $[n]$ satisfying $\sigma_i > \sigma_{i+1}$ 
for any odd index $i$ with $1\le i<n$, 
$\sigma_i > \sigma_{i-1}$ for any odd index $i$ with $1<i\le n$ 
and the extra condition $\sigma_1 > \sigma_{n}$, since in $\Lambda_n$ 
we have $b_1\cdot b_n=0$. This last requirement on $\sigma$ is 
easily verified by tracing the first appearance of a 1 in the first and the last columns of the 
table of paths that define the bijection for $\Gamma_n$. Therefore, 
$\sigma$ must be a circular alternating permutation, 
and these were enumerated by Kreweras \cite{Kreweras}.

For $n$ odd, assume that  $u_1=0(01)^{\frac{n-1}{2}}$ and $v_1=0(10)^{\frac{n-1}{2}}$. Then we know that $u_1,v_1\in 0\Gamma_{n-1}$ and since $\Lambda_n=0\Gamma_{n-1}+10\Gamma_{n-3}0$ we have
$$c(u_1,v_1;\Lambda_n)=c(u_1,v_1;0\Gamma_{n-1})=E_{n-1}\,. $$
Let $u_i$ and $v_i$ be the $i-1$ right cyclic shifts of the vertices $u_1$ and $v_1$ for $i=2,\ldots,n$, respectively. Then for any shortest path $P$ from $u_1$ to $v_1$, the $i-1$ right cyclic shifts of all the vertices in $P$ gives a shortest path from $u_i$ to $v_i$ 
for all $i\in\{2,\ldots,n\}$, which completes the proof.
\end{proof}

\section{Calculation for the Alternate Lucas cubes}
For any integer $n\ge 3$, it is shown in \cite{ALucas} that $\diam(\LL_n)=n-1$. We have 
\begin{proposition}\label{prop.ALucas.max.distance.vertices}
	For any integer $n\ge 4$, the number of diametrically opposite pair of vertices in $\LL_n$ is $4$. For $n\ge 4$, they are
	\begin{itemize}
		\item[(i)] $u=0^s(10)^{k}001$ and $v=1^s(01)^{k}010$, 
		\item[(ii)] $u=0^s(10)^{k}010$ and $v=1^s(01)^{k}001$,
		\item[(iii)] $u=0^s(10)^{k}100$ and $v=1^s(01)^{k}001$, 
		\item[(iv)] $u=0^s(10)^{k}100$ and $v=1^s(01)^{k}010$,
	\end{itemize}
	where $n=2k+3+s$, $k$ is a nonnegative integer and $s\in\{0,1\}$.
\end{proposition}
\begin{theorem}\label{theorem_cn_ALucas}
	Let $n=2k+3+s$, $k$ be a nonnegative integer and $s\in\{0,1\}$. For $ n \geq 4$, we have
	$$c(0^s(10)^{k}100,1^s(01)^{k}001;\LL_n)=c(0^s(10)^{k}100,1^s(01)^{k}010;\LL_n)=E_{n-1}$$
	$$c(0^s(10)^{k}001,1^s(01)^{k}010;\LL_n)=c(0^s(10)^{k}010,1^s(01)^{k}001;\LL_n)={n-1 \choose 2}E_{n-3} \, .$$
\end{theorem}
\begin{proof}
We sketch the proof. 
As in the proof of Theorem \ref{theorem_cn_Fibo}, 
we need to consider the permutations $\sigma$ of $[n]$ satisfying extra conditions depending on the pair of vertices.
We will give the proof for $n$ even ($s=1$) and only for the pairs 
$u=0^s(10)^{k}100$ and $v=1^s(01)^{k}001$ and 
$u=0^s(10)^{k}001$ and $v=1^s(01)^{k}010$. The other cases can be obtained similarly.

For the pair $u=0(10)^{k}100$ and $v=1(01)^{k}001$ as we consider the shortest paths we will not 
change the $(n-1)$st position since it is 0 for each vertex. Therefore we need to consider the permutations $\sigma$ of $[n] \setminus\{n-1\}$ satisfying $\sigma_i > \sigma_{i+1}$ for any odd index $i$ with $1\le i\le n-3$, 
$\sigma_i > \sigma_{i-1}$ for any odd index $i$ with $1<i\le n-3$ and $\sigma_{n} > \sigma_{n-2}$,
 since in $\LL_n$ we have $b_{n-2}\cdot b_n=0$. By setting $\tau_i=\sigma_i$ for $i=1, \ldots, n-2$ and $\tau_{n-1}=\sigma_n$ we observe that $\tau$ is an 
alternating permutation of $[ n-1]$.

Now consider the 
pair $u=0(10)^{k}001$ and $v=1(01)^{k}010$. In the shortest paths 
under consideration,
we will not change the $(n-2)$nd position since it is 0 for each vertex. Therefore we need to consider the 
permutations $\sigma$ of $[n] \setminus\{n-2\}$ satisfying $\sigma_i > \sigma_{i+1}$ 
for any odd index $i$ with $1\le i< n-3$, 
$\sigma_i > \sigma_{i-1}$ for any odd index $i$ with $1<i\le n-3$ 
and $\sigma_{n-1} > \sigma_{n}$. By setting $\tau_i=\sigma_i$ for $i=1, \ldots, n-3$ we observe that $\tau$ is an 
alternating permutation of $[n-3 ]$ and we have ${n-1 \choose 2}$ different choices for $\sigma_{n-1}, \sigma_{n}$ which gives the desired result.
\end{proof} 
\section*{Acknowledgement}
The work of the second author is supported by BAP-SUK-2021-19737 of Hacettepe University. This work is partially supported by T\"{U}B\.ITAK under Grant No. 120F125.

\end{document}